\documentclass{amsart}
\usepackage{amssymb}

\newtheorem{theorem}{Theorem}[section]
\newtheorem{lemma}[theorem]{Lemma}
\newtheorem{corollary}[theorem]{Corollary}
\theoremstyle{definition}
\newtheorem{definition}[theorem]{Definition}

\DeclareMathOperator{\diam}{diam}

\DeclareMathOperator{\cl}{cl}
\DeclareMathOperator{\inter}{int}
\DeclareMathOperator{\Perm}{Perm}

\begin{document}

\title{Dirichlet sets and Erd\"os-Kunen-Mauldin theorem}
\author{Peter Elia\v s}
\address{Mathematical Institute, Slovak Academy of Sciences,
  Gre\v s\'akova 6, 040\,01 Ko\v sice, Slovakia}
\email{elias@upjs.sk}
\thanks{This work was supported by grant No.~1/3002/26 of VEGA Grant Agency.}

\begin{abstract}
  By a theorem proved by Erd\"os, Kunen and Mauldin,
  for any nonempty perfect set $P$ on the real line there exists
  a perfect set $M$ of Lebesgue measure zero such that $P+M=\mathbb{R}$.
  We prove a stronger version of this theorem
  in which the obtained perfect set $M$ is a Dirichlet set.
  Using this result we show that for a wide range of familes of subsets of the reals,
  all additive sets are perfectly meager in transitive sense.
  We also prove that every proper analytic subgroup $G$ of the reals
  is contained in an $F_\sigma$-set $F$ such that $F+G$ is a meager null set.
\end{abstract}

\maketitle

\section{Introduction}

In 1981, P. Erd\"os, K. Kunen and R. D. Mauldin proved that
if $P$ is a nonempty perfect subset of $\mathbb{R}$ then there exists a perfect set $M$
of Lebesgue measure zero such that the sum $P+M=\{x+y:x\in P,\ y\in M\}$ is the
whole real line \cite{EKM}.
Their proof was based on a variation of a number-theoretic theorem of G. G. Lorentz.
We now present a different proof, based on the Kronecker's theorem.
We also obtain a somewhat stronger result, finding a perfect set $M$ that is a Dirichlet set.
Later, we will use this result to prove theorems on small sets of reals
related to harmonic analysis.

The notion of Dirichlet sets was introduced by J.-P. Kahane in the late 1960's
(\cite{Kahane1}, see also \cite{Kahane2}, \cite{Lindahl-Poulsen}).
By the original definition, a set $E\subseteq\mathbb{R}$ is a Dirichlet set
if there exists a sequence $\{\lambda_n\}_{n\in\mathbb{N}}$ of real numbers such that
$\lim_{n\to\infty}\left|\lambda_n\right|=\infty$ and the sequence of functions
$\left\{e^{i\lambda_nx}\right\}_{n\in\mathbb{N}}$ converges uniformly to $1$ on $E$.
For our purpose, it is much  easier to use the following equivalent definition, cf.~\cite{KahaneS},
\cite{BuKhRe}.
We will denote by $\left\|x\right\|$ the distance of a real $x$ to the nearest integer,
i.e., $\left\|x\right\|=\min\{\left|x-k\right|:k\in\mathbb{Z}\}$.

\begin{definition}
  A set $E\subseteq\mathbb{R}$ is called a {\em Dirichlet set}\/ if there exists
  an increasing sequence of natural numbers $\{n_k\}_{k\in\mathbb{N}}$ such that
  for all $x\in E$ and for all $k\in\mathbb{N}$, $\left\|n_kx\right\|\le 2^{-k}$.
\end{definition}

Dirichlet sets were introduced to help the study of other types of so-called
`thin sets of harmonic analysis' (see \cite{BuKhRe}), e.g., N-sets, known also as `sets of absolute convergence',
or Arbault sets (called also A-sets) introduced by J. Arbault \cite{Arbault} under the name `sets admitting
a sequence with zero limit'.
One can define these types of sets as follows.

\begin{definition}
  A set $E\subseteq\mathbb{R}$ is an {\em Arbault set}\/ (an {\em A-set}\/)
  if there exists an  increasing sequence of natural numbers $\{n_k\}_{k\in\mathbb{N}}$
  such that $\lim_{k\to\infty}\left\|n_kx\right\|=0$ for all $x\in E$.
  It is an {\em N-set}\/ if there exists a sequence
  $\{a_n\}_{n\in\mathbb{N}}$ of non-negative
  reals such that $\sum_{n\in\mathbb{N}}a_n=\infty$ and for all $x\in E$,
  the sum $\sum_{n\in\mathbb{N}}a_n\left\|nx\right\|<\infty$.
\end{definition}

Let $\mathcal{D}$, $\mathcal{A}$, and $\mathcal{N}$ denote the families of all
Dirichlet sets, A-sets and N-sets, respectively.
Every Dirichlet set is both A-set and N-set.
It is well known (see \cite{Arbault}, \cite{BuKhRe}) that the families $\mathcal{A}$ and $\mathcal{N}$
are both included in the intersection of the ideals of meager and Lebesgue null sets.
If $E$ is an A-set (an N-set) then any subset of the subgroup $\langle E\rangle$ of $(\mathbb{R},+)$
generated by $E$, is again an A-set (an N-set, respectively).
On the other side, $\mathcal{A}$ and $\mathcal{N}$ are not closed under unions.
Also, no inclusion between the families $\mathcal{A}$, $\mathcal{N}$ holds true.

Let us note that Dirichlet sets, N-sets and A-sets can equally be considered as subsets of the
circle group $\mathbb{T}=\mathbb{R}/\mathbb{Z}$.

J. Arbault and independently P. Erd\"os (see \cite[p. 271]{Arbault}) proved that a union of
N-set and a countable set is again an N-set.
An analogous result for Arbault sets was proved by N. N. Kholshchevnikova.

This inspired the following definition \cite{Arbault}.

\begin{definition}[Arbault]
 A set $A\subseteq\mathbb{R}$ is called {\em permitted\/} if $A\cup E$ is an N-set
 for any N-set $E$.
\end{definition}

Thus, Arbault-Erd\"os theorem says that every countable set is permitted.
In \cite{Arbault}, J. Arbault also claimed that there exists a perfect permitted set.
However, N. K. Bari \cite{Bari} later found a gap in his proof and the question of the existence of
perfect permitted sets remained open.

Let us note that since the family $\mathcal{N}$ is closed under generating of the subgroups,
the union $A\cup E$ in the definition of permitted sets may be replaced by the sum $A+E$.

In \cite{Lafontaine}, J. Lafontaine tried to show that there is no perfect permitted set.
He claimed that for any perfect set $P$ there exists an N-set $X$ such that the sum
$X+P$ has a positive measure.
However, Lafontaine's proof seems to contain a gap too; he proves that the set $X+P$
is positive with respect to some Borel measure $\mu$ which is a convolution of two other measures,
but there is no evidence why $\mu$ should be the standard Lebesgue measure.
Without this, $X+P$ may still be an N-set (which must be Lebesgue null), and
$P$ may still be a permitted set.

In 1990s, several results were published in which the existence of uncountable permitted sets was proved
under various set-theoretic assumptions (see \cite{BuKhRe}).

\section{A strengthening of Erd\"os-Kunen-Mauldin theorem}

Lafontaine's claim and the theorem of Erd\"os, Kunen and Mauldin are closely related.
Both state that for any perfect set $P$ there is a small set $X$ such that the set
$X+P$ is large.
When considering Dirichlet sets as `small' sets and the whole real line as the only `large'
set, we obtain the following strengthening of both statements.

\begin{theorem} \label{thm:EKM-Dirichlet}
  For any perfect set $P\subseteq\mathbb{R}$ there exists a Dirichlet set
  $D\subseteq\mathbb{R}$ such that $P+D=\mathbb{R}$.
\end{theorem}

This theorem can be formulated and proved the same way also for the group of $\mathbb{T}$
in place of $\mathbb{R}$.

In the proof of Theorem~\ref{thm:EKM-Dirichlet}, we will use the following well-known result of approximation theory
\cite{Cassels, HR}.

\begin{theorem}[Kronecker's Theorem]
 Let real numbers $x_1,\dots,x_n$ be linearly independent over\/ $\mathbb{Q}$, and let
 $y_1,\dots,y_n$ be arbitrary.
 Then for every $\varepsilon>0$ there exists
 $k\in\mathbb{Z}$ such that for every $i$,
 $\left\|kx_i-y_i\right\|<\varepsilon$.
\end{theorem}

Theorem \ref{thm:EKM-Dirichlet} follows from the following lemma,
which will also be useful later.
Quantifier $\forall^\infty$ means `for all but finitely many'.

\begin{lemma} \label{lem:EKM-Dirichlet}
 Let $P\subseteq\mathbb{R}$ be a perfect set.
 Then there exists an increasing sequence of natural numbers $\{n_k\}_{k\in\mathbb{N}}$
 such that
 \begin{enumerate}
  \item[(a)] for every $y\in\mathbb{R}$ there exists $p\in P$ such that\/
   $\forall k\,\left\|n_k(p-y)\right\|\le 2^{-k}$,
  \item[(b)] for every $y\in\mathbb{R}$, the set $P\cap(y-D)$ is dense in $P$, where
   $D=\{x\in\mathbb{R}:\forall^\infty k\,\left\|n_kx\right\|\le 2^{-k}\}$.
 \end{enumerate}
\end{lemma}

\begin{proof}
 Fix a countable set $Q=\{q_k:k\in\mathbb{N}\}$ dense in $P$.
 For every $k\in\mathbb{N}$, put
 $B_k=\left\{\frac{m}{2^{k+1}}:m\in\mathbb{Z},\ 0\le m<2^{k+1}\right\}$.
 By induction, we define $n_k$, $\varepsilon_k$ and $A_k$ such that
 for every $k$, $n_{k+1}>n_k$, $\varepsilon_{k+1}\le{\varepsilon_k}/{2}$,
 $A_k$ is a finite subset of $P$ containing $q_k$,
 and for every $a\in A_k$ and $b\in B_k$ there is $a'\in A_{k+1}$
 such that
 \begin{enumerate}
  \item[1.] $\left|a'-a\right|<\displaystyle\frac{\varepsilon_k}{2}$, and
  \item[2.] for all $x$, if $\left|x-a'\right|<\varepsilon_{k+1}$ then
   $\left\|n_{k+1}x-b\right\|<2^{-(k+2)}$.
 \end{enumerate}

 This can be done as follows.
 Let $A_0=\{q_0\}$ and $\varepsilon_0>0$ be arbitrary.
 For every $a\in A_k$ and $b\in B_k$, pick $p_k(a,b)\in P$ so that
 $\left|p_k(a,b)-a\right|<{\varepsilon_k}/{2}$, and the set
 $\{p_k(a,b):a\in A_k,\ b\in B_k\}$ is linearly independent over $\mathbb{Q}$.
 Using Kronecker's theorem find a natural number $n_{k+1}>n_k$ such that
 for all $a\in A_k$ and $b\in B_k$,
 $\left\|n_{k+1}\,p_k(a,b)-b_k\right\|<2^{-(k+2)}$.
 There exists a positive real $\varepsilon_{k+1}\le{\varepsilon_k}/{2}$ such that
 for all $a\in A_k$, $b\in B_k$, and for all $x$,
 if $\left|x-p_k(a,b)\right|\le\varepsilon_{k+1}$ then
 $\left\|n_{k+1}x-b_k\right\|<2^{-(k+2)}$.
 Put $A_{k+1}=\{q_0,\dots,q_k\}\cup\{p_k(a,b):a\in A_k,b\in B_k\}$.

 Let us prove part (b) first.
 Let $D=\{x\in\mathbb{R}:\forall^\infty k\,\left\|n_kx\right\|\le 2^{-k}\}$,
 and let $y$ be a given real.
 We are going to show that $P\cap(y-D)$ is dense in $P$.
 Let $U$ be an open set such that $P\cap U\neq\emptyset$.
 For $k\in\mathbb{N}$, take $b_k\in B_k$ such that
 $\left\|n_{k+1}y-b_k\right\|\le 2^{-(k+2)}$.
 There exists $k_0$ such that the ball $B_{q_{k_0}}(\varepsilon_{k_0})$ is a subset of $U$.
 We put $p_{k_0}=q_{k_0}$ and by induction choose $p_{k+1}\in A_{k+1}$, for $k\ge k_0$,
 such that $\left|p_{k+1}-p_k\right|<\frac{\varepsilon_k}{2}$, and for all $x$,
 if $\left|x-p_{k+1}\right|<\varepsilon_{k+1}$ then $\left\|n_{k+1}x-b_k\right\|<2^{-(k+2)}$.
 Clearly $p_k\to p$ for some $p\in P$, and for all $k\ge k_0$,
 $\left|p-p_k\right|<\varepsilon_k$.
 Thus $p\in U$ and if $k\ge k_0$ then
 \begin{displaymath}
  \left\|n_{k+1}(y-p)\right\|\le
  \left\|n_{k+1}p-b_{k+1}\right\|+\left\|n_{k+1}y-b_{k+1}\right\|\le
  2^{-(k+1)},
 \end{displaymath}
 hence $y-p\in D$, i.e., $p\in y-D$.

 Part (a) can be proved the same way starting with $U=\mathbb{R}$ and $k_0=0$.
\end{proof}

We can now prove Theorem~\ref{thm:EKM-Dirichlet}.

\begin{proof}[Proof of Theorem \ref{thm:EKM-Dirichlet}]
 For a given perfect set $P$, let $\{n_k\}_{k\in\mathbb{N}}$ be the sequence found
 in Lemma \ref{lem:EKM-Dirichlet}, and let $D$ be Dirichlet set
 $\left\{x:\forall k\,\left\|n_kx\right\|<2^{-k}\right\}$.
 By (a), for every $y\in\mathbb{R}$ there exists $p\in P$ such that $p-y\in D$.
 Since $D$ is symmetric, we also have $y-p\in D$, and hence $y=p+(y-p)\in P+D$.
\end{proof}

\begin{corollary} \label{cor:no-perfect}
 There is no perfect permitted set.
\end{corollary}

\begin{proof}
 Let $P$ be a perfect set.
 By Theorem \ref{thm:EKM-Dirichlet}, there exists a Dirichlet set $D$ such that $P+D=\mathbb{R}$.
 If $P$ would be permitted then $P\cup D$, and also the subgroup $\langle P\cup D\rangle$
 generated by $P\cup D$,
 would be N-sets.
 However, we have $\langle P\cup D\rangle\supseteq P+D=\mathbb{R}$, a contradiction.
\end{proof}

\section{Permitted sets for various families of sets}

The result from Corollary~\ref{cor:no-perfect} can be formulated in a more general way.
In order to do it, we extend the notion of permitted sets from $\mathcal{F}$ to other families of sets.

\begin{definition}
    We call a family of sets $\mathcal{F}$ {\em hereditary\/} if a subset of any member
    of $\mathcal{F}$ is a member of $\mathcal{F}$ too.
    Let $\mathcal{F}$ be a hereditary family of subsets of a group.
    We say that a set $A$ is {\em permitted for $\mathcal{F}$} if for any $B\in\mathcal{F}$
    also $A+B\in\mathcal{F}$.
    We denote by $\Perm(\mathcal{F})$ the family of all sets permitted for $\mathcal{F}$.
\end{definition}

Let us note that $\Perm(\mathcal{F})$ is a hereditary family closed under finite unions, i.e., an ideal.
Sometimes, permitted sets are called {\em additive}, e.g., in the case when $\mathcal{F}$ is the family
of all meager or null subsets of the real line.

We are now going to prove a stronger version of Corollary \ref{cor:no-perfect}.
To formulate it, we will need the following notions.

\begin{definition}
     A subset $A$ of a topological space $X$ is called {\em perfectly meager}\/
     if for any perfect set $P\subseteq X$, the set $A\cap P$ is meager in the relative topology of $P$.
\end{definition}

The notion of perfectly meager sets has some natural modifications,
see  \cite{Nowik-Weiss}, \cite{Zakrzewski}.

\begin{definition}[Zakrzewski]
    A set $A$ is called {\em universally meager}\/ if for any countable collection $\mathcal{P}$
    of perfect subsets of $X$ there exists an $F_\sigma$-set $F\supseteq A$ such that
    $F\cap P$ is meager in $P$ for every $P\in\mathcal{P}$.
\end{definition}

\begin{definition}[Nowik, Weiss]
    Let $X$ be a topological group.
    A set $A\subseteq X$ is called {\em perfectly meager in transitive sense}\/
    if for any perfect set $P\subseteq X$ there exists an $F_\sigma$-set $F\supseteq A$
    such that for every $y\in X$, $(F+y)\cap P$ is meager in $P$.
\end{definition}

In topological groups $\mathbb{R}$ and $\mathbb{T}$, every set which is perfectly meager in transitive sense
is also unversally meager and every unviversally meager set is perfectly meager.
The existence of sets contradicting the opposite implications is known to be
consistent with ZFC \cite{Nowik-Weiss}, \cite{Zakrzewski}.

We will need one more family of trigonometric thin sets.

\begin{definition}
 A set $E\subseteq\mathbb{R}$ is called {\em pseudo-Dirichlet set}\/
 if there is an increasing sequence of natural numbers $\{n_k\}_{k\in\mathbb{N}}$
 such that for all $x\in E$, $\forall^\infty k\,\left\|n_kx\right\|<2^{-k}$.
 Let $p\mathcal{D}$ denote the family of all pseudo-Dirichlet sets.
\end{definition}

Let us note that $\mathcal{D}\subseteq p\mathcal{D}\subseteq\mathcal{A}\cap\mathcal{N}$,
where both inclusions are proper (see, e.g., \cite{BuKhRe}, \cite{KahaneS}).

\begin{theorem} \label{thm:PMTS}
 Let $\mathcal{F}$ be a hereditary family of subsets of\/ $\mathbb{R}$ such that
 $p\mathcal{D}\subseteq\mathcal{F}$ and\/ $\mathbb{R}\notin\mathcal{F}$.
 Assume that for every set $E\in\mathcal{F}$ there exists an $F_\sigma$-set $F$
 such that $E\subseteq F$ and $E+F\neq\mathbb{R}$.
 Then every set permitted for the family $\mathcal{F}$ is perfectly meager in transitive sense.
\end{theorem}

\begin{proof}
 Let $A\in\Perm(\mathcal{F})$, and let $P$ be a perfect set.
 By Lemma \ref{lem:EKM-Dirichlet} (b), there exists a pseudo-Dirichlet set $D$ such that
 for all $y\in\mathbb{R}$, $P\cap(y-D)$ is dense in $P$.
 We have $D+A\in\mathcal{F}$, and hence there exists an $F_\sigma$-set $F\supseteq D+A$
 such that $D+F\neq\mathbb{R}$.
 It follows that $F\cap(x-D)=\emptyset$ for some $x\in\mathbb{R}$,
 and hence $P\setminus(y+F)\supseteq P\cap(x+y-D)$ is dense in $P$ for every $y\in\mathbb{R}$.
 Since $F$ is $F_\sigma$, $P\cap(y+F)$ is meager in $P$ for every $y\in P$.
\end{proof}

\begin{corollary} \label{cor:PMTS-trigfam}
 Every set permitted for some of the familes $p\mathcal{D}$, $\mathcal{N}$, $\mathcal{A}$
 is perfectly meager in transitive sense.
\end{corollary}

\begin{proof}
 It suffices to show that if $\mathcal{F}$ is any of the families $p\mathcal{D}$, $\mathcal{N}$, $\mathcal{A}$
 then for every set $E\in\mathcal{F}$ there exists an $F_\sigma$-set $F$
 such that $E\subseteq F$ and $E+F\neq\mathbb{R}$.
 This is evident for $p\mathcal{D}$ and $\mathcal{N}$, since every N-set is contained in an $F_\sigma$
 N-set which is a proper subgroup of $\mathbb{R}$.

 In the case of $\mathcal{A}$ we need a different argument.
 For every Arbault set $E$ there exists an increasing sequence $\{n_k\}_{k\in\mathbb{N}}$ such that
 $\lim_{k\to\infty}\left\|n_kx\right\|=0$ for every $x\in E$.
 Put $F=\left\{x\in\mathbb{R}:\forall^\infty k\,\left\|n_kx\right\|\le\frac{1}{8}\right\}$.
 Clearly $F$ is an $F_\sigma$-set, $E\subseteq F$,
 and $F+F\subseteq H=\left\{x\in\mathbb{R}:\forall^\infty k\,\left\|n_kx\right\|\le\frac{1}{4}\right\}$.
 We have $H\neq\mathbb{R}$ since the Lebesgue measure of $H\cap \left[0,1\right]$ is not greater than
 $\frac{1}{2}$.
\end{proof}

In the next part of the paper we will prove that the assumptions of Theorem~\ref{thm:PMTS}
are satisfied for a wide range of families of subsets of $\mathbb{T}$ and $\mathbb{R}$.

\section{Analytic subgroups of the reals}

In this section we will use the standard set-theoretic notation:
$\omega$ denotes the set of all natural numbers,
$\omega^\omega$ and $\omega^{<\omega}$
denote the sets of all infinite and of all finite sequences of natural numbers, respectively.
For $t,s\in\omega^{<\omega}$, $t^\smallfrown s$ denotes the concatenation of $t$ and $s$;
$t\subseteq s$ means that $t$ is a initial segment of $s$.
For $x\in\omega^\omega$ and $n\in\omega$,
$x\upharpoonright n$ is the initial segment of $x$ of the length $n$.

A subset of a Polish space
(i.e., a separable, completely metrizable topological space)
is called {\em analytic\/} if it is a continuous image of a Borel subset
of some other Polish space.
It is well known that a set $A$ is analytic if and only if there exists a {\em Suslin scheme for $A$},
i.e., an indexed family $\{A_t:t\in\omega^{<\omega}\}$ such that every $A_t$ is a closed set,
$A_t\subseteq A_s$ whenever $t\supseteq s$, and
$A=\bigcup_{x\in\omega^\omega}\bigcap_{n\in\omega} A_{x\upharpoonright n}$.
We will say that Suslin scheme $\{A_t:t\in\omega^{<\omega}\}$ has {\em vanishing diameters\/}
if for every $x\in\omega^\omega$, $\lim_{n\to\infty}\diam A_{x\upharpoonright n}=0$.

Let $\mathcal{E}$ denote the $\sigma$-ideal generated by closed sets of Lebsgue measure zero.
Clearly all sets contained in $\mathcal{E}$ are both meager and null.

M. Laczkovich \cite{Laczkovich} proved that every proper analytic subgroup of the reals is a subset of
an $F_\sigma$ null set.
We will show a result which is substantially stronger.

\begin{theorem} \label{thm:analytic-null}
  For any proper analytic subgroup $A$ of\/ $\mathbb{R}$ or\/ $\mathbb{T}$ there exists an $F_\sigma$-set
  $F\supseteq A$ such that $A+F\in\mathcal{E}$.
\end{theorem}

To prove Theorem \ref{thm:analytic-null}, we will use the following result of Solecki \cite{Solecki}.
By a portion of a set we mean a nonempty relatively open subset.

\begin{lemma} \label{lem:Solecki}
  Let $A$ be an analytic set in a Polish space and let $\mathcal{I}$ be
  any $\sigma$-ideal generated by closed sets.
  Then either $A\in\mathcal{I}$ or there is a $G_\delta$-set $G\subseteq A$ such that
  no portion of $G$ belongs to $\mathcal{I}$.
\end{lemma}

We will also need some lemmas.
We will consider only subsets and subgroups of the reals (i.e., $\mathbb{R}$ or $\mathbb{T}$).
However, the results are valid in a more general setting.
In Lemma~\ref{lem:analytic-1} and \ref{lem:ccp} we will assume
a Polish group with invariant metric and with $\sigma$-ideal of null sets defined as follows:
a set $A$ is {\em null}\/ if for every $\varepsilon>0$ there exists a sequence
of open balls $\{B_n\}_{n\in\omega}$ such that $A\subseteq\bigcup_{n\in\omega}B_n$
and $\sum_{n\in\omega}\diam B_n<\varepsilon$.
Thus, a set is null iff its 1-dimensional Hausdorff measure is zero.
For subsets of $\mathbb{R}$ and $\mathbb{T}$, this is equivalent to having standard Lebesgue
measure zero.

\begin{lemma} \label{lem:analytic-1}
  Let $A$ be an analytic null set, and let $\varepsilon>0$.
  Then there exists a Suslin scheme $\{B_t:t\in\omega^{<\omega}\}$ for the set $A$ such that\/
  $\sum_{t\in\omega^{<\omega}\setminus\{\emptyset\}}\diam B_t<\varepsilon$.
\end{lemma}

\begin{proof}
  Let $\{A_t:t\in\omega^{<\omega}\}$ be an arbitrary Suslin scheme for $A$.
  Put $T=\omega^{<\omega}\setminus\{\emptyset\}$.
  Since $T$ is countable and $A$ is a null set,
  for any $t\in T$ there is a sequence $\big\{I^t_j\big\}_{j\in\omega}$ of closed sets
  such that $A\subseteq\bigcup_{j\in\omega}I^t_j$,
  and $\sum_{t\in T}\sum_{j\in\omega}\diam I^t_j<\varepsilon$.
  Let us fix a bijection $\varphi:\omega\times\omega\to\omega$.

  We will define $\{B_t:t\in\omega^{<\omega}\}$ and a function
  $\psi:\omega^{<\omega}\to\omega^{<\omega}$ as follows.
  Put $B_\emptyset=A_\emptyset$ and $\psi(\emptyset)=\emptyset$.
  We proceed for all $t\in\omega^{<\omega}$ by induction.
  If $B_t$ and $\psi(t)$ are already defined then we have $B_t\subseteq A_{\psi(t)}$.
  For every $k,j\in\omega$ put
  $B_{t^\smallfrown\varphi(k,j)}=B_t\cap A_{\psi(t)^\smallfrown k}\cap I^{t^\smallfrown k}_j$
  and $\psi(t^\smallfrown\varphi(k,j))=\psi(t)^\smallfrown k$.

  Let $B=\bigcup_{y\in\omega^\omega}\bigcap_{n\in\omega}B_{y\upharpoonright n}$.
  It is easy to see that $\{B_t:t\in\omega^{<\omega}\}$ is a Suslin scheme,
  $\sum_{t\in T}\diam B_t<\varepsilon$, and $B\subseteq A$.

  To see that $A\subseteq B$, assume that $a\in\bigcap_{n\in\omega}A_{x\upharpoonright n}$
  for some $x\in\omega^\omega$.
  Let us define $y\in\omega^\omega$ inductively as follows.
  For $n\in\omega$, take $j_n$ such that $a\in I^{y\upharpoonright n^\smallfrown x(n)}_{j_n}$
  and put $y(n)=\varphi(x(n),j_n)$.
  Then for every $n\in\omega$ we have $\psi(y\upharpoonright n)=x\upharpoonright n$ and
  \begin{math}
    B_{y\upharpoonright n+1}=
    B_{y\upharpoonright n}\cap
    A_{x\upharpoonright n+1}\cap
    I^{y\upharpoonright n^\smallfrown x(n)}_{j_n},
  \end{math}
  hence $a\in\bigcap_{n\in\omega}B_{y\upharpoonright n}$.
\end{proof}

\begin{definition}
  We say that a set $A$ {\em can be covered by countably many copies of $B$}\/
  if  there exists a countable set $C$ such that $A\subseteq B+C$.
  We say that a family $\{A_t:t\in\omega^{<\omega\}}$ has
  {\em countable covering property\/} if for every $t\in\omega^{<\omega}$ there 
  exists a finite set $F_t\subseteq\omega^{<\omega}$ such that for every
  $s\in\omega^{<\omega}\setminus F_t$, the set $A_s$ can be covered by countably
  many copies of $A_t$.
\end{definition}

\begin{lemma} \label{lem:ccp}
  Let $A$ be a meager analytic null set in a Polish group.
  Then there exists a Suslin scheme $\{A_t:t\in\omega^{<\omega}\}$ for $A$ having
  vanishing diameters and countable covering property.
\end{lemma}

\begin{proof}
  Let $T=\omega^{<\omega}\setminus\{\emptyset\}$ and let $\varepsilon>0$ be arbitrary.
  Since $A$ is analytic set of Lebesgue measure zero, by Lemma~\ref{lem:analytic-1} we can find
  a Suslin scheme $\{B_t:t\in\omega^{<\omega}\}$ for $A$ such that
  $\sum_{t\in\omega^{<\omega}\setminus\{\emptyset\}}\diam B_t<\varepsilon$.
  Since $A$ is meager, we may assume that $B_t$ is nonempty and nowhere dense, for every
  $t\in T$.

  Let $\{t_n\}_{n\in\omega}$ be an enumeration of the set $T$ such that for any $m,n\in\omega$,
  if $t_m\subseteq t_n$ then $m\le n$.
  There exists a sequence $\{c_n\}_{n\in\omega}$
  such that the sets $B_{t_n}+c_n$ are pairwise disjoint,
  and if $b_n\in B_{t_n}+c_n$ for every $n$ then
  the sequence $\{b_n\}_{n\in\omega}$ converges to $0$.
  Hence, for every $n$, the set $C_n=\bigcup_{k\ge n}(B_{t_k}+c_k)\cup\{0\}$ is closed.

  Fix $m\in\omega$.
  Let $(m)\in\omega^1$ denote the sequence containing only one element, namely $m$.
  Since the set $B_{(m)}$ is nowhere dense, there exists a pairwise disjoint family
  $\left\{B\!\left(d^m_j,\delta^m_j\right):j\in\omega\right\}$
  of closed balls with center $d^m_j$ and radius $\delta^m_j$ such that
  $B_{(m)}\subseteq\cl\{d^m_j:j\in\omega\}$,
  and if $b_j\in B\!\left(d^m_j,\delta^m_j\right)$ for every $j$ then
  any accumulation point of the sequence $\{b_j\}_{j\in\omega}$ belongs to $B_{(m)}$.

  For every $t\in T$ such that $t(0)=m$, let $J_t\subseteq\omega$ be a nonempty set such that
  if $b_j\in B\!\left(d^m_j,\delta^m_j\right)$ for every $j\in J_t$ then
  the set of all limit points of $\{b_j:j\in J_t\}$ is exactly the set $B_t$.
  We may assume that $J_s\subseteq J_t$ for $s\supseteq t$.
  Put
  \begin{displaymath}
    E_t=B_t\cup\bigcup_{j\in J_t}B\!\left(d^m_j,\delta^m_j\right)\cap(C_n+d^m_j),
  \end{displaymath}
  where $n$ is such that $t=t_n$.
  Clearly $E_t$ is closed.

  Let $\left\{H_t:t\in\omega^{<\omega}\right\}$ be a family of closed sets
  such that for all $t$ and $s$,
  $B_t\subseteq\inter H_t$, $\diam H_t<2\diam B_t$, and
  $H_s\subseteq H_t$ whenever $s\supseteq t$.

  We define family $\left\{A_t:t\in\omega^{<\omega}\right\}$ as follows.
  Let $A_\emptyset=B_\emptyset$ (= the whole space), and for all $t\in T$ let
  $A_t=H_t\cap E_t$.
  It is easy to see that $\left\{A_t:t\in\omega^{<\omega}\right\}$
  is a Suslin scheme with vanishing diameters and that
  $\bigcup_{x\in\omega^\omega}\bigcap_{n\in\omega}A_{x\upharpoonright n}=A$.

  To show that $\{A_t:t\in\omega^{<\omega}\}$ has countable covering property, let us take
  $t\in\omega^{<\omega}$.
  If $t=\emptyset$ then $A_t$ clearly covers every $A_s$, $s\in\omega^{<\omega}$.
  If $t\in T$ then let $m=t(0)$ and $t=t_n$.
  There exists $j\in J_t$ such that $d^m_j\in\inter H_t$.
  There also exists $n'\ge n$ such that for every $k\ge n'$,
  $B_{t_k}+c_k+d^m_j$ is a subset of $H_t\cap B\left(d^m_j,\delta^m_j\right)$,
  and hence also of $A_t=H_t\cap E_t$.
  Since $A_{t_k}$ can be covered by a union of countably many translations of sets $B_{t_l}$ for $l\ge k$,
  every set $A_{t_k}$ such that $k\ge n'$ can be covered by countable many copies of $A_t$.
\end{proof}

\begin{proof}[Proof of Theorem~\ref{thm:analytic-null}]
  Let $A$ be a proper analytic subgroup of $\mathbb{R}$ (the same proof will work for $\mathbb{T}$).
  Since $A$ is meager and Lebesgue null, by Lemma~\ref{lem:ccp} there exists a Suslin scheme
  $\{A_t:t\in\omega^{<\omega}\}$ for $A$ having vanishing diameters and posessing the countable covering property.
  Let $\mathcal{F}$ be the family of all closed sets $F\subseteq\mathbb{R}$ such that
  there exists some $t\in\omega^{<\omega}$ for which $F+A_t$ is a null set,
  and let $\mathcal{I}$ denote the $\sigma$-ideal generated by $\mathcal{F}$.
  By Lemma~\ref{lem:Solecki}, either $A\in\mathcal{I}$ or there exists a $G_\delta$-set $G\subseteq A$
  such that no portion of $G$ belongs to $\mathcal{I}$.

  In the first case we have $A\subseteq F$, where $F=\bigcup_{n\in\omega}F_n$ for some $F_n\in\mathcal{F}$,
  hence $F$ is an $F_\sigma$-set and for every $n$ there is $t_n\in\omega^{<\omega}$
  such that $F_n+A_{t_n}\in\mathcal{E}$.
  From the countable covering property it follows that for every $n$ there exists $k_n$
  such that if $t\in\omega^{<\omega}$ has length at least $k_n$ then $F_n+A_t$
  can be covered by countable  many copies of $F_n+A_{t_n}$
  and hence $F_n+A_t\in\mathcal{E}$.
  Thus, $F_n+A\in\mathcal{E}$ for every $n$, and $F+A\in\mathcal{E}$.

  In the rest of the proof we show that the second case is impossible.
  Assume that $G\subseteq A$ is a $G_\delta$-set such that if $G\cap U\neq\emptyset$
  for some $U$ open then $G\cap U\notin\mathcal{I}$.
  Let $F$ be a closure of $G$.
  For every open set $U$, if $F\cap U\neq\emptyset$ then $F\cap U\notin\mathcal{I}$ and hence
  $(F\cap U)+A_t$ has positive measure for every $t\in\omega^{<\omega}$.

  We will use the same trick as does the proof of Lemma~2 in \cite{Laczkovich}.
  Let $I_0$ be an arbitrary non-trivial interval in $\mathbb{R}$.
  We shall play a Banach-Mazur game in $I_0$ with the second player winning if the intersections
  of intervals played is a subset of $G+G+A$.
  We will provide a winning strategy for the second player and this way prove that the set
  $G+G+A$ is comeager in $I_0$.
  Since $G+G+A\subseteq A$ and $A$ is a proper analytic subgroup of $\mathbb{R}$,
  hence a meager set, we obtain a contradiction.

  Let $G=\bigcup_{i\in\omega}G_i$ where $G_i$ are open dense subsets of $F$.
  Fix an arbitrary $h\in\omega^\omega$.
  Assume that the first move of the first player is an interval $I_1\subseteq I_0$.
  Since $A_\emptyset=\mathbb{R}$, $I_0$ is trivially contained in the set $F+A_\emptyset+F+A_\emptyset$.
  Since $G_1$ is dense in $F$, there exist $x_1,y_1\in G_1$ and $a_1,b_1\in A_\emptyset$
  such that $x_1+a_1+y_1+b_1\in\inter I_1$.
  There also exist open intervals $J_1,K_1$ such that $x_1\in J_1$, $y_1\in K_1$,
  $\cl J_1\subseteq G_1$, $\cl K_1\subseteq G_1$, and $J_1+a_1+K_1+b_1\subseteq I_1$.
  Sets $(F\cap J_1)+A_{(0)}$, $(F\cap K_1)+A_{(0)}$ are of positive measure,
  hence its sum $(F\cap J_1)+A_{(0)}+(F\cap K_1)+A_{(0)}$ contains a non-trivial closed interval $I_2$.
  Let $I_2$ be the response of the second player.

  On the $k^{\textrm{th}}$ move, let the first player plays an interval $I_{2k-1}\subseteq I_{2k-2}$
  where $I_{2k-2}$ is an interval contained in
  $(F\cap J_{k-1})+A_{h\upharpoonright{(k-1)}}+(F\cap K_{k-1})+A_{h\upharpoonright{(k-1)}}$.
  Since $G_k$ is dense in $F$, there exist $x_k\in G_k\cap J_{k-1}$, $y_k\in G_k\cap K_{k-1}$,
  and $a_k,b_k\in A_{h\upharpoonright{(k-1)}}$ such that $x_k+a_k+y_k+b_k\in\inter I_{2k-1}$.
  Let $J_k$, $K_k$ be open intervals such that
  $x_k\in J_k$, $y_k\in K_k$, $\cl J_k\subseteq G_k$, $\cl K_k\subseteq G_k$, and
  $J_k+a_k+K_k+b_k\subseteq I_{2k-1}$.
  Sets $(F\cap J_k)+A_{h\upharpoonright k}$, $(F\cap K_k)+A_{h\upharpoonright k}$
  are of positive measure, hence its sum contains a non-trivial closed interval $I_{2k}$.
  Let $I_{2k}$ be the move of the second player.

  Now, let $x\in\bigcap_{k\in\omega}J_k$, $y\in\bigcap_{k\in\omega}K_k$,
  and let $a\in\bigcap_{k\in\omega}A_{h\upharpoonright k}$.
  Then $a\in A$, $a_k\to a$, $b_k\to a$, and $x,y\in G$.
  If $z$ is in the intersection of all intervals $I_k$, $k\in\omega$, then
  $z=x+a+y+a$, hence $z\in G+G+A$.
\end{proof}

\begin{corollary}
  Let $A$ be a proper analytic subgroup of\/ $\mathbb{R}$ or\/ $\mathbb{T}$.
  Then $A$ can be separated from one of its cosets by an $F_\sigma$-set.
\end{corollary}

\begin{proof}
  By Theorem~\ref{thm:analytic-null}, there exists an $F_\sigma$-set $F\supseteq A$ such that
  $A+F$ is null.
  Hence, there exists $x\in\mathbb{T}\setminus F+A$, and thus
  $F$ is disjoint with the set $x-A=x+A$.
\end{proof}

We do not know whether such $F_\sigma$-set does exist for any coset of $A$.

\section{Permitted sets for families generated by analytic subgroups}

In this short section we bring a version of Theorem~\ref{thm:PMTS} for families generated
by analytic subgroups of $\mathbb{R}$.
The results are equally valid also for the group $\mathbb{T}$.

From Theorem~\ref{thm:analytic-null} it follows that for every proper analytic subgroup
$A$ of $\mathbb{R}$ there exists an $F_\sigma$-set $F$ such that $A+F\neq\mathbb{R}$.
Hence the assumption of Theorem~\ref{thm:PMTS} is satisfied for every family generated
by proper analytic subgroups.

\begin{corollary}
  Let $\mathcal{F}$ be a hereditary family generated by some collection of
  proper analytic subgroups of\/ $\mathbb{R}$ such that
  $p\mathcal{D}\subseteq\mathcal{F}$.
  Then every set permitted for the family $\mathcal{F}$ is perfectly meager
  in transitive sense.
\end{corollary}

We show that a family generated by subgroups
contains all pseudo-Dirichlet sets if it contains all Dirichlet sets.

\begin{lemma}
  A set $E$ is a pseudo-Dirichlet set iff it is a subset of a group generated by a Dirichlet set.
\end{lemma}

\begin{proof}
  If $D$ is a Dirichlet set then
  $D\subseteq\{x:\forall k\ \left\|n_k\,x\right\|\le 2^{-k}\}$
  for some increasing sequence of natural numbers $\{n_k\}_{k\in\mathbb{N}}$.
  Clearly the group generated by $D$ is
  $\langle D\rangle=\{x:\exists m\ \forall k\ \left\|n_k\,x\right\|\le m.2^{-k}\}$.
  Hence if $x\in\langle D\rangle$ then there exists $k_0$ such that
  $\forall k>k_0\ \left\|n_k\,x\right\|\le 2^{-k+1}$,
  and $\langle D\rangle$ is a pseudo-Dirichlet set.
  
  On the other hand, let
  $E\subseteq\{x:\forall^\infty k\ \left\|n_k x\right\|\le 2^{-k}\}$
  be a pseudo-Dirichlet set.
  If $x\in E$ then there exists $k_0$ such that
  $\forall k>k_0\ \left\|n_k\,x\right\|\le 2^{-k}$.
  For all $k$ we have $\left\|n_k\,x\right\|\le 2^{k_0}.2^{-k}$,
  and hence $x\in\langle D\rangle$ where
  $D=\{x:\forall k\ \left\|n_k\,x\right\|\le 2^{-k}\}$ is a Dirichlet set.
\end{proof}

As a corollary we obtain the following.

\begin{theorem}
  Let $\mathcal{F}$ be a hereditary family generated by some collection of proper
  analytic subgroups of\/ $\mathbb{R}$ such that every Dirichlet set is in $\mathcal{F}$.
  Then every set permitted for the family $\mathcal{F}$ is perfectly meager
  in transitive sense.
\end{theorem}

We do not know whether the assumption concerning Dirichlet sets
can be omitted or weakened.

\end{document}